\documentclass[a4paper, DIV=12, 11pt]{scrartcl}
\usepackage[latin1]{inputenc}
\usepackage{amsmath}
\usepackage{amsfonts}
\usepackage{amssymb}
\usepackage{amsthm}
\usepackage{graphicx}
\usepackage{thmtools}
\usepackage{bbm}
\usepackage{hyperref}
\usepackage[bottom, marginal]{footmisc}
\usepackage{todonotes}

\usepackage{tikz}
\usetikzlibrary{matrix}

\declaretheoremstyle[headfont=\normalfont]{normalhead}
\newtheorem{lemma}{Lemma}[section]
\newtheorem{theorem}[lemma]{Theorem}
\newtheorem{proposition}[lemma]{Proposition}
\newtheorem{corollary}[lemma]{Corollary}

\newtheorem{question}{Question}

\newcounter{mt}

\newcommand{\R}{\mathbb{R}}
\newcommand{\C}{\mathbb{C}}

\newcommand{\U}{\mathrm{U}}

\DeclareMathOperator{\Val}{Val}

\DeclareMathOperator{\vol}{vol}

\DeclareMathOperator{\supp}{supp}

\DeclareMathOperator{\GL}{GL}

\DeclareMathOperator{\GW}{\mathrm{GW}}

\DeclareMathOperator{\Sym}{\mathrm{Sym}}

\DeclareMathOperator{\tr}{\mathrm{tr}}

\setkomafont{title}{\normalfont\Large}

\author{Jonas Knoerr}
\title{Smooth valuations on convex bodies and finite linear combinations of mixed volumes}
\date{}

\newcommand{\Addresses}{{
		\bigskip
		\footnotesize
		
		Jonas Knoerr, \textsc{Institute of Discrete Mathematics and Geometry, TU Wien, Wiedner Hauptstrasse 8-10, 1040 Wien, Austria}\par\nopagebreak
		\textit{E-mail address}: \texttt{jonas.knoerr@tuwien.ac.at}
		
		\medskip
	}}
	
\makeatletter
\def\blfootnote{\xdef\@thefnmark{}\@footnotetext}
\makeatother

\makeindex
\begin{document}
\maketitle
\begin{abstract}
	It is shown that Alesker's solution of McMullen's conjecture implies the following stronger version of the conjecture: Every continuous, translation invariant, $k$-homogeneous valuation on convex bodies in $\R^n$ can be approximated uniformly on compact subsets by finite linear combinations of  mixed volumes involving at most $N_{n,k}$ summands, where $N_{n,k}$ is a constant depending on $n$ and $k$ only.
	Moreover, $n-k-1$ of the arguments of the mixed volumes can be chosen to be ellipsoids that do not depend on the valuation. The result is based on a corresponding description of smooth valuations in terms of finite linear combinations of mixed volumes.
\end{abstract}

\section{Introduction}
	A real-valued map $\mu$ defined on the space of convex bodies $\mathcal{K}(\R^n)$, that is, the set of all non-empty, compact, and convex subsets of $\R^n$ equipped with the Hausdorff metric, is called a \emph{valuation} if
	\begin{align*}
		\mu(K)+\mu(L)=\mu(K\cup L)+\mu(K\cap L)
	\end{align*}
	for all $K,L\in\mathcal{K}(\R^n)$ such that $K\cup L$ is convex. The class of valuations includes many important geometric functionals, for example Euler characteristic, surface area, and volume, as well as mixed volumes. Here, the mixed volume $V(K_1,\dots,K_{n})$ of convex bodies $K_1,\dots,K_n\in\mathcal{K}(\R^n)$ is given by the corresponding coefficient of the polynomial
	\begin{align*}
		\vol_n(\lambda_1K_1+\dots+\lambda_nK_n)=\sum_{i_1,\dots,i_n=1}^n \lambda_{i_1}\dots\lambda_{i_n} V(K_{i_1},\dots,K_{i_n}),\quad \lambda_1,\dots,\lambda_n\ge 0,
	\end{align*}
	where $\vol_n$ denotes the Lebesgue measure on $\R^n$. More precisely, $K\mapsto V(K[k],L_1,\dots,L_{n-k})$ defines a continuous and translation invariant valuation on $\mathcal{K}(\R^n)$ for each choice of reference bodies $L_1,\dots,L_{n-k}\in\mathcal{K}(\R^n)$, where $K$ is taken with multiplicity $k$.\\
	
	Following a number of breakthrough results in the 1990s, the theory of valuations on convex bodies has developed rapidly, with far reaching generalizations of classical results and new applications in integral geometry and other disciplines of mathematics. For an overview of these developments we refer to Alesker and Fu \cite{AleskerFuIntegralgeometryvaluations2014}, and Bernig \cite{BernigAlgebraicintegralgeometry2011}.	Many of these developments were sparked by Alesker's solution of a conjecture by McMullen, which states that mixed volumes span a dense subspace of the space $\Val(\R^n)$ of all continuous and translation invariant valuations on $\mathcal{K}(\R^n)$ with respect to the topology of uniform convergence on compact subsets in $\mathcal{K}(\R^n)$.\\
	
	Alesker obtained an affirmative solution to McMullen's conjecture from a much more general description of the space $\Val(\R^n)$. Let $\Val_k(\R^n)$ denote the subspace of $\Val(\R^n)$ of $k$-homogeneous valuations, that is, the space of all $\mu\in\Val(\R^n)$ that satisfy $\mu(tK)=t^k\mu(K)$ for all $K\in\mathcal{K}(\R^n)$, $t\ge0$. By a result by McMullen \cite{McMullenValuationsEulertype1977}, we have a direct sum decomposition 
	\begin{align*}
		\Val(\R^n)=\bigoplus_{k=0}^n\Val_k(\R^n).
	\end{align*}
	This decomposition implies in particular that $\Val(\R^n)$ is a Banach space with respect to the norm
	\begin{align*}
		\|\mu\|:=\sup_{K\subset B_1(0)}|\mu(K)|,
	\end{align*}
	which metrizes the topology of uniform convergence on compact subsets (note that the set $\{K\in\mathcal{K}(\R^n):K\subset B_1(0)\}$ is compact due to the Blaschke Selection Theorem). We have a natural continuous operation $\pi$ of $\GL(n,\R)$ on $\Val(\R^n)$ given by
	\begin{align*}
		[\pi(g)\mu](K)=\mu(g^{-1}K)\quad\text{for}~K\in\mathcal{K}(\R^n),~g\in\GL(n,\R).
	\end{align*}
	Under this operation, the space $\Val_k(\R^n)$ further decomposes into a direct sum $\Val_k^+(\R^n)\oplus\Val_k^-(\R^n)$ of closed subspaces of even/odd valuations, where $\mu\in\Val(\R^n)$ is called even/odd if $\mu(-K)=\pm\mu(K)$ for $K\in\mathcal{K}(\R^n)$. We have the following representation theoretic description of $\Val(\R^n)$, which is known as the \emph{Irreducibility Theorem}. 
	\begin{theorem}[Alesker \cite{AleskerDescriptiontranslationinvariant2001}]
		\label{theorem:IrredicibilityTheorem}
		The natural representation of $\GL(n,\R)$ on $\Val_k^{\pm}(\R^n)$ is irreducible.
	\end{theorem}
	Here, a continuous representation of $\GL(n,\R)$ is called irreducible if every non-trivial, $\GL(n,\R)$-invariant subspace is dense. This result directly implies McMullen's conjecture, as it is not difficult to see that the space spanned by mixed volumes is $\GL(n,\R)$-invariant and intersects $\Val_k^\pm(\R^n)$ non-trivially. As pointed out by Alesker \cite{AleskerDescriptiontranslationinvariant2001}, one can even restrict the class of reference bodies $L_1,\dots,L_{n-k}$ in the mixed volumes $K\mapsto V(K[k],L_1,\dots,L_{n-k})$ to smaller $\GL(n,\R)$-invariant families of convex bodies, for example ellipsoids in the case of even valuations, or simplices in the general case.\\

	In this article we consider the following refinement of McMullen's conjecture in terms of the size of the linear combination of mixed volumes required for the approximation. 
	\begin{theorem}
	\label{theorem:densityFiniteComb}
	Let $N:=\binom{n+1}{2}+1$. There exist ellipsoids $\mathcal{E}_{1},\dots,\mathcal{E}_{N}$ such that for every $\mu\in\Val_k(\R^n)$ there exist sequences $(L^\pm_{\alpha,j})_j$ of smooth convex bodies with strictly positive Gauss curvature, $\alpha\in \mathbb{N}^N$ with $|\alpha|=n-k-1$,  such that 
	\begin{align*}
		\mu(K)=\lim\limits_{j\rightarrow\infty}\sum_{\substack{\alpha\in \mathbb{N}^{N}\\|\alpha|=n-k-1}}V(K[k],L^+_{\alpha,j},\mathcal{E}_{1}[\alpha_1],\dots,\mathcal{E}_{N}[\alpha_N])-V(K[k],L^-_{\alpha,j},\mathcal{E}_{1}[\alpha_1],\dots,\mathcal{E}_{N}[\alpha_N])
	\end{align*}
	holds uniformly on compact subsets of $\mathcal{K}(\R^n)$. In particular, every such valuation can be approximated by finite linear combinations of at most $2\binom{\binom{n+1}{2}+n-k-1}{n-k-1}$ mixed volumes.
	\end{theorem}
	Our approach is based on a corresponding description of the space of \emph{smooth valuations} obtained by Alesker in \cite{Aleskermultiplicativestructurecontinuous2004} as a consequence of Theorem \ref{theorem:IrredicibilityTheorem}. Recall that $\mu\in\Val_k(\R^n)$ is called smooth if
	\begin{align*}
		\GL(n,\R)&\rightarrow \Val_k(\R^n)\\
		g&\mapsto \pi(g)\mu
	\end{align*}
	is a smooth map. Standard facts from representation theory imply that the subspace $\Val_k(\R^n)^{sm}$ of smooth $k$-homogeneous valuations is dense in $\Val_k(\R^n)$. This space is naturally equipped with the topology of a Fr\'echet space that is finer than the subspace topology, and the restriction of the operation of $\GL(n,\R)$ on $\Val(\R^n)^{sm}$ is continuous with respect to this topology. Moreover, Alesker's proof of Theorem \ref{theorem:IrredicibilityTheorem} applies to this representation as well, that is, the spaces $\Val_k^\pm(\R^n)^{sm}$ of smooth even/odd valuations of degree $k$ are irreducible with respect to the Fr\'echet topology. In fact, for the type of representation considered in \cite{AleskerDescriptiontranslationinvariant2001}, this is equivalent to Theorem \ref{theorem:IrredicibilityTheorem} (which can be deduced, for example, from \cite{HarishChandraRepresentationssemisimpleLie1953} Theorem 5).\\
	
	The main observation of this article is that McMullen's conjecture holds for smooth valuations in a much stronger form - every such valuation is a finite linear combination of mixed volumes. Let us remark that this was also independently observed by van Handel using a very similar idea, and we thank him for the productive discussion of this problem.\\
	More precisely, we have the following representation of smooth valuations in terms of mixed volumes.
	\begin{theorem}
		\label{theorem:finiteCombMixedVolumes}
		There exist $N:=\binom{n+1}{2}+1$ ellipsoids $\mathcal{E}_{1},\dots,\mathcal{E}_{N}$ such that for every $\mu\in\Val_k(\R^n)^{sm}$ there exist smooth convex bodies $L_\alpha^\pm\in \mathcal{K}(\R^n)$ with strictly positive Gauss curvature, $\alpha\in \mathbb{N}^N$ with $|\alpha|=n-k-1$, such that for $K\in\mathcal{K}(\R^n)$,
		\begin{align*}
			\mu(K)=\sum_{\substack{\alpha\in \mathbb{N}^N\\|\alpha|=n-k-1}}V(K[k],L^+_\alpha,\mathcal{E}_{1}[\alpha_1],\dots,\mathcal{E}_{N}[\alpha_N])-V(K[k],L^-_\alpha,\mathcal{E}_{1}[\alpha_1],\dots,\mathcal{E}_{N}[\alpha_N]).
		\end{align*}
		Moreover, for each $\alpha$ one of the bodies $L^+_\alpha$ and $L^-_\alpha$ can be chosen to be an ellipsoid, and if $\mu$ is an even valuation, then the remaining bodies can be chosen to be origin symmetric.
	\end{theorem}
	Note in particular that the ellipsoids $\mathcal{E}_1,\dots,\mathcal{E}_N$ do not depend on the valuation. Counting the number of indices, we obtain that any smooth valuation in $\Val_{k}(\R^n)$ can be represented as a linear combinations of at most $2\binom{\binom{n+1}{2}+n-k-1}{n-k-1}$ mixed volumes. As smooth valuations are dense in $\Val_k(\R^n)$, this directly implies Theorem \ref{theorem:densityFiniteComb}.\\
	
	In light of this result, one might be tempted to ask whether mixed volumes with arbitrary convex bodies span $\Val_k(\R^n)$ for any $0\le k\le n$. However, it is not difficult to see that this is not the case, apart from the trivial cases $k=0$ and $k=n$, which are spanned by the Euler characteristic and the Lebesgue measure (as shown by Hadwiger \cite{HadwigerVorlesungenuberInhalt1957}) respectively. For all other degrees we give an explicit counterexample in Section \ref{section:ContinuousCase} based on previous work by Goodey and Weil \cite{GoodeyWeilDistributionsvaluations1984}, and McMullen \cite{McMullenContinuoustranslationinvariant1980}.\\

	Let us comment on the proof of Theorem \ref{theorem:finiteCombMixedVolumes}. In \cite{Aleskermultiplicativestructurecontinuous2004}, Alesker gave a description of smooth valuations that can be interpreted as a representation of these functionals as a convergent series of linear combinations of mixed volumes. More precisely, his construction extends the mixed volume to a multilinear functional on smooth functions on a product of spheres and uses the Schwartz Kernel Theorem to obtain a linear map into the space of smooth valuations. By combining the Irreducibility Theorem \ref{theorem:IrredicibilityTheorem} (or rather the corresponding statement for the representations $\Val^\pm_k(\R^n)^{sm}$) with the Casselman-Wallach Theorem \cite{CasselmanCanonicalextensionsHarish1989}, he showed that this map is onto, which provides the desired representation. However, the mixed volumes can also be interpreted in terms of \emph{mixed area measures}, which reduces the construction to an integral over the unit sphere of the mixed discriminant of the Hessians of the functions involved. Up to a technical lemma guaranteeing the convergence, the multilinearity of the mixed discriminant thus reduces the proof of Theorem \ref{theorem:finiteCombMixedVolumes} to choosing ellipsoids such that the Hessians of their support functions span the space of symmetric matrices at each point of the sphere.\\
	
	This article is structured as follows. We first discuss a well known lemma related to the Schwartz Kernel Theorem and recall some basic facts about mixed area measures in Section \ref{section:preliminaries}. We then prove Theorem \ref{theorem:finiteCombMixedVolumes} in Section \ref{section:ProofMainResult}. In Section \ref{section:ContinuousCase} we show that mixed volumes span a proper subspace of $\Val_k(\R^n)$ for all $1\le k\le n-1$. Section \ref{section:questions} contains some comments and questions related to the results.
	
	\paragraph{Acknowledgments}
	I want to thank Andreas Bernig, Jan Kotrbat\'y, and Ramon van Handel for the discussions during the preparation of this article.
\section{Preliminaries}
\label{section:preliminaries}
\subsection{A lemma related to the Schwartz Kernel Theorem}
For a smooth manifold $S$, let $C^\infty(S)$ denote the space of smooth, $\C$-valued functions on $S$ and $C_c^\infty(S)$ the subspace of function with compact support. We will need the following well-known result.
\begin{lemma}
	\label{lemma:SchwartzKernelDecompFunction}
	Let $X_1,\dots,X_k$ be compact manifolds of dimensions $n_1,\dots,n_k$. There exist constants $C_M>0$, $M\in\mathbb{N}$, with the following property: For every smooth function $F\in C^\infty(X_1\times\dots\times X_k)$ there exist smooth functions $f_i^j\in C^\infty(X_i)$, $1\le i\le k$, $j\in\mathbb{N}$, such that
	\begin{enumerate}
		\item $F=\sum_{j=1}^{\infty}f_1^j\otimes\dots \otimes f_k^j$ converges in the $C^\infty$-topology, and
		\item for all $l=(l_1,\dots,l_k)\in\mathbb{N}^k$ the following holds (with $M_l:=|l|+\sum_{i=1}^{k}n_i+1$): 
			\begin{align*}
				\sum_{j=1}^{\infty}\prod_{i=1}^k \|f_i^j\|_{C^{l_i}(X_i)}\le C_{M_l}\|F\|_{C^{M_l}(X_1\times\dots\times X_k)}.
			\end{align*}
	\end{enumerate}
\end{lemma}		
Let us remark that most references state this result in weaker forms that are not sufficient for our application, although the version given above is contained in the proof. We will therefore include a proof for the convenience of the reader. Using a partition of unity, the result directly follows from the following lemma.
\begin{lemma}
	For  $1\le i\le k$ let $K_i\subset\R^{n_i}$ be a compact subset and $U_i$ a bounded open neighborhood of $K_i$. There exist constants $C_M>0$, $M\in\mathbb{N}$, depending on $M$ and $K_1,\dots,K_k$, $U_1,\dots,U_k$ only, such that the following holds: For every $F\in C_c^\infty(U_1\times\dots\times U_k)$ that is supported on $K_1\times\dots\times K_k$ there exist functions $f_i^j\in C_c^\infty(U_i)$, $1\le i\le k$, $j\in\mathbb{N}$, such that
	\begin{enumerate}
		\item $F=\sum_{j=1}^{\infty}f_1^j\otimes\dots \otimes f_k^j$, and
		\item for all $l=(l_1,\dots,l_k)\in\mathbb{N}^k$ the following holds (with $M_l:=|l|+\sum_{i=1}^{k}n_i+1$): 
		\begin{align*}
			\sum_{j=1}^{\infty}\prod_{i=1}^k\|f_i^j\|_{C^{l_i}(U_i)}\le C_{M_l}\|F\|_{C^{M_l}(U_1\times\dots\times U_k)}.
		\end{align*}
	\end{enumerate}
\end{lemma}
\begin{proof}
	We follow the argument of Gask given in \cite{GaskproofSchwartzskernel1961} Section 2. Assume that $U_i$ is contained in a centered cube with side length $2a$. Let $A_i$ denote the corresponding open cube with side lengths $2a$ in $\R^{n_i}$ for $1\le i\le k$. For $p_i\in\mathbb{Z}^{n_i}$ we consider the functions on $A_i$ given by
	\begin{align*}
		e^i_{p_i}(x_i)=\gamma_{n_i}\exp\left(i\frac{\pi}{a}\langle p_i,x_i\rangle\right), \quad x_i\in A_i,
	\end{align*}
	where $\gamma_{n_i}$ is a constant depending on $n_i$ only such that these functions are pairwise orthonormal. Decompose $F$ into a Fourier series on $A_1\times\dots\times A_k$,
	\begin{align*}
		F=\sum_{p_1,\dots,p_k}a_{p_1,\dots,p_k}e^1_{p_1}\otimes\dots\otimes e^k_{p_k}.
	\end{align*}
	Take $\psi_i\in C_c(U_i)$ with $\psi_i\equiv 1$ on a neighborhood of $K_i$. Setting $f^i_{p_i}= \psi_i e^i_{p_i}$, we obtain
	\begin{align*}
		F=\sum_{p_1,\dots,p_k}a_{p_1,\dots,p_k}f^1_{p_1}\otimes\dots\otimes f^k_{p_k}.
	\end{align*}
	The coefficients $a_{p_1,\dots,p_k}\in \C$ are given by well known integral formulas, and integration by parts shows that there exist constants $D_j>0$ for $j\in\mathbb{N}$ independent of $F$ such that 
	\begin{align*}
		|a_{p_1,\dots,p_k}|\le D_j\|F\|_{C^j(U_1\times\dots\times U_k)}\left(1+\sum_{i=1}^{k}|p_i|\right)^{-j}.
	\end{align*}
	Moreover, for every $M\in\mathbb{N}$ there exists a constant $E_M>0$ such that
	\begin{align*}
		\|f^i_{p_j}\|_{C^M(U_i)}\le E_M (1+|p_j|)^{M}
	\end{align*}
	by the Leibniz rule. Consequently, for $l=(l_1,\dots,l_k)\in\mathbb{N}^k$,
	\begin{align*}
		&\sum_{p_1,\dots,p_k}|a_{p_1,\dots,p_k}|\prod_{i=1}^k\|f^i_{p_j}\|_{C^{l_i}(U_i)}\le\left(\max_{M\le \max l_i} E_M\right)^k \sum_{p_1,\dots,p_k}|a_{p_1,\dots,p_k}|\prod_{i=1}^k(1+|p_i|)^{l_i}\\
		\le&\left(\max_{M\le \max l_i} E_M\right)^k D_{1+\sum_{i}|l_i|+n_i}\|F\|_{C^{1+\sum_{i}|l_i|+n_i}(U_1\times\dots\times U_k)}\sum_{p_1,\dots,p_k}\left(1+\sum_{i=1}^{k}|p_i|\right)^{-\left(1+\sum_{i}n_i\right)},
	\end{align*}
	where the sum on the right hand side converges. The claim follows by setting
	\begin{align*}
		C_{1+\sum_{i}|l_i|+n_i}:=\left(\max_{M\le \max l_i} E_M\right)^k D_{1+\sum_{i}|l_i|+n_i}\sum_{p_1,\dots,p_k}\left(1+\sum_{i=1}^{k}|p_i|\right)^{-\left(1+\sum_{i}n_i\right)}
	\end{align*}
	and replacing $f^1_{p_1,\dots,p_k}$ by $a_{p_1,\dots,p_{k}}f^1_{p_1,\dots,p_k}$.
\end{proof}
	\subsection{Mixed area measures of convex bodies}
	We refer to \cite{SchneiderConvexbodiesBrunn2014} for a background on convex bodies and area measures.
	The surface area measure $S_{n-1}(K)$ of a convex body $K\in\mathcal{K}(\R^n)$ can be characterized as the unique Radon measure on $S^{n-1}$ such that 
	\begin{align}
		\label{eq:defSurfaceAreaMeasure}
		\frac{d}{dt}\Big|_0\vol_n(K+tL)=\int_{S^{n-1}}h_L dS_{n-1}(K)\quad\text{for all}~L\in\mathcal{K}(\R^n).
	\end{align}
	Here, $h_K(x):=\sup_{x\in K}\langle x,y\rangle$, $x\in \R^n$, denotes the support function of $K\in\mathcal{K}(\R^n)$, which is a $1$-homogeneous convex function. In particular, $K\mapsto S_{n-1}(K)$ is a translation invariant, measure-valued valuation on $\mathcal{K}(\R^n)$. If $K$ is smooth with strictly positive Gauss curvature, then its support function restricts to a smooth function on $S^{n-1}$ and $S_{n-1}(K)$ is absolutely continuous with respect to the Hausdorff measure $\mathcal{H}^{n-1}$ on $S^{n-1}$ with
	\begin{align}
		\label{eq:surfaceAreaMeasureSmoothBody}
		dS_{n-1}(K,x)=\det\nolimits_{n-1}(D^2h_K(x))d\mathcal{H}^{n-1}(x),
	\end{align}
	where $D^2f(x)$ denotes the restriction of the Hessian of the $1$-homogeneous extension of $f\in C^2(S^{n-1})$ to the tangent space $T_xS^{n-1}$. Note that this is related to the spherical Hessian $\nabla^2f$ by $D^2f=\nabla^2f+fId$.\\
	As the surface area measure is continuous with respect to weak* convergence, one can easily deduce from this representation that for $L_1,\dots,L_{n-1}\in\mathcal{K}(\R^n)$ the map
	\begin{align*}
		(\lambda_1,\dots,\lambda_{n-1})\mapsto S_{n-1}\left(\sum_{i=1}^{n-1}\lambda_iL_i\right)
	\end{align*}
	is a homogeneous polynomial of degree $n-1$ in $\lambda_1,\dots,\lambda_{n-1}\ge 0$. The map
	\begin{align*}
		(L_1,\dots,L_{n-1})\mapsto S_{n-1}(L_1,\dots,L_{n-1}):=\frac{1}{(n-1)!}\frac{\partial^{n-1}}{\partial\lambda_1\dots\partial\lambda_{n-1}}\Big|_0S_{n-1}\left(\sum_{i=1}^{n-1}\lambda_iL_i\right)
	\end{align*}
	is called the \emph{mixed area measure} and depends continuously on $L_1,\dots,L_{n-1}$ in the weak* topology. Note that there exists a constant $c_{n,k}\in\R$ independent of $K,L_1,\dots,L_{n-k}$ such that
	\begin{align*}
		S_{n-1}(K[k],L_1,\dots,L_{n-k-1})=c_{n,k}\frac{\partial^{n-k-1}}{\partial\lambda_1\dots\partial\lambda_{n-k-1}}\Big|_0S_{n-1}\left(K+\sum_{i=1}^{n-k-1}\lambda_iL_i\right),
	\end{align*}
	so $K\mapsto S_{n-1}(K[k],L_1,\dots,L_{n-k-1})$ defines a valuation with values in the space of signed Radon measures on $S^{n-1}$. It is easy to see that this map is translation invariant and continuous with respect to the weak* topology. If $K, L_1,\dots,L_{n-k}$ are smooth convex bodies with strictly positive Gauss curvature, then \eqref{eq:surfaceAreaMeasureSmoothBody} implies that the mixed area measure is absolutely continuous with respect to the Hausdorff measure with
	\begin{align}
		\label{eq:MixedsurfaceAreaMeasureSmooth}
		&dS_{n-1}(K[k],L_1,\dots,L_{n-k-1})
		=D_{n-1}(D^2h_K[k],D^2h_{L_1},\dots,D^2h_{L_{n-k-1}})d\mathcal{H}^{n-1},
	\end{align}
	where $D_{n-1}$ denotes the mixed discriminant.\\
	Note that \eqref{eq:defSurfaceAreaMeasure} implies that the mixed volume can be expressed in terms of the mixed area measure: For $K,L_1,\dots,L_{n-k}\in\mathcal{K}(\R^n)$,
	\begin{align}
		\label{eq:relationMixedVolumeMixedSurfaceAreaMeasure}
		V(K[k],L_1,\dots,L_{n-k})=\frac{1}{n}\int_{S^{n-1}}h_{L_1}dS_{n-1}(K[k],L_2,\dots,L_{n-k}).
	\end{align}
	
\section{Proof of Theorem \ref{theorem:finiteCombMixedVolumes}}
	\label{section:ProofMainResult}
	In \cite{Aleskermultiplicativestructurecontinuous2004} Section 1, Alesker provided a construction of smooth valuations in terms of convergent sums of mixed volumes. His construction was formulated in invariant terms, however, the result can be stated in the following way.
	\begin{proposition}[Alesker \cite{Aleskermultiplicativestructurecontinuous2004}]
		\label{proposition:ThetaMixedVolumes}
		There exists a continuous linear map $\Theta_{n-k,0}:C^\infty((S^{n-1})^{n-k})\rightarrow \Val_k(\R^n)^{sm}\otimes \C$ such that
		\begin{align*}
			\Theta_{n-k,0}(h_{L_1}\otimes\dots\otimes h_{L_{n-k}})[K]=V(K[k],L_1,\dots,L_{n-k})\quad\text{for}~\quad K\in\mathcal{K}(\R^n),
		\end{align*}
		for all smooth convex bodies $L_1,\dots,L_{n-k}$ with strictly positive Gauss curvature. Moreover, $\Theta_{n-k,0}$ is onto. 
	\end{proposition}
	Let us remark that the proof given in \cite{AleskerDescriptiontranslationinvariant2001} contains a minor inaccuracy: The fact that this map is onto is deduced from the Casselman-Wallach Theorem \cite{CasselmanCanonicalextensionsHarish1989}, which shows that the image of this map is closed, and the Irreducibility Theorem \ref{theorem:IrredicibilityTheorem}, which implies that the image dense. However, due to the different topologies, it is not sufficient to apply the Irreducibility Theorem for the spaces $\Val_k^\pm(\R^n)$; instead one has to use that the spaces $\Val_k^\pm(\R^n)^{sm}$ of smooth valuations are irreducible with respect to their natural Fr\'echet topology. As mentioned in the introduction, these two different notions are essentially equivalent for this type of representation.\\

	Combining this result with Lemma \ref{lemma:SchwartzKernelDecompFunction}, it is not difficult to see that any smooth valuation admits a representation as a convergent sum of mixed volumes. We will require the following reinterpretation of this construction.
	\begin{corollary}
		\label{corollary:smoothValSeriesMixedVolumes}
		There exists a continuous linear map $\Theta_{n-k}:C^\infty((S^{n-1})^{n-k})\rightarrow \Val_k(\R^n)^{sm}\otimes\C$ such that
		\begin{align*}
			\Theta_{n-k}(h_{L_1}\otimes\dots\otimes h_{L_{n-k}})[K]=\int_{S^{n-1}}h_{L_1}dS_{n-1}(K[k],L_2,\dots,L_{n-k}), \quad K\in\mathcal{K}(\R^n),
		\end{align*}
		for all smooth convex bodies $L_1,\dots,L_{n-k}$ with strictly positive Gauss curvature. Moreover, $\Theta_{n-k}$ is onto.
	\end{corollary}	
	\begin{proof}
		Due to the relation of mixed volumes and mixed area measures in \eqref{eq:relationMixedVolumeMixedSurfaceAreaMeasure} and Proposition \ref{proposition:ThetaMixedVolumes}, we can take $\Theta_{n-k}:=n\Theta_{n-k,0}$.
	\end{proof}
	Next, we are going to construct the relevant ellipsoids. For $A\in \Sym^2(\R^n)$ positive definite, let $\mathcal{E}_A\in\mathcal{K}(\R^n)$ denote the ellipsoid with support function $h_{\mathcal{E}_{A}}(x)=\sqrt{\langle x,Ax\rangle}$ for $x\in\R^n$. Consider the matrix $E_{ij}\in \Sym^2(\R^n)$ corresponding to the quadratic form $x\mapsto x_ix_j-\delta_{ij}x_ix_j$ and set $A_{ij}(t):=tId+E_{ij}$ for $t\ge 1$. 
	\begin{lemma}
		\label{lemma:ChoiceEllipsoids}
		There exists $t\ge 1$ such that $D^2h_{\mathcal{E}_{A_{ij}(t)}}(x)$, $1\le i\le j\le n$, and $D^2h_{\mathcal{E}_{Id}}(x)$ span $\Sym^2 (T_xS^{n-1})$ for all $x\in S^{n-1}$.
	\end{lemma}
	\begin{proof}
	Recall that $(C_1,C_2)\mapsto \tr(C_1C_2)$ defines a scalar product on $\Sym^2(\R^n)$. As $E_{ij}$, $1\le i\le j\le n$, is a basis of $\Sym^2(\R^n)$, we can thus define a norm on $\Sym^2(\R^n)$ by $\|C\|_*:=\max_{i,j}|\tr(C E_{ij})|$. Let us denote the operator norm of $C\in \Sym^2(\R^n)$ by $\|C\|$. We can then find $c>0$ such that $c\|C\|\le \|C\|_*$ for all $C\in \Sym^2(\R^n)$. We claim that $t:=1+\frac{2}{c}$ has the desired properties.\\
		
	First, a simple calculation shows that for any positive definite matrix $A\in \Sym^2(\R^n)$, the Hessian of $h_{\mathcal{E}_A}(x)=\sqrt{\langle x,Ax\rangle}$ is given by
	\begin{align*}
		\mathrm{Hess}(h_{\mathcal{E}_A})(x)=\frac{\langle x,Ax\rangle A-(Ax)(Ax)^T}{\sqrt{\langle x,Ax\rangle}^3}.
	\end{align*}
	We will show that the restrictions of the matrices 
	\begin{align*}
		B_{ij}(x):=&	\sqrt{\langle x,A_{ij}(t)x\rangle}^3\mathrm{Hess}(h_{\mathcal{E}_{A_{ij}(t)}})(x)-t\langle x,A_{ij}(t)x\rangle \mathrm{Hess}(h_{\mathcal{E}_{Id}})(x)\\
		=&\langle x,A_{ij}(t)x\rangle E_{ij}-(A_{ij}(t)x)\cdot (A_{ij}(t)x)^T+t\langle x,A_{ij}(t)x\rangle x\cdot x^T\\
		=&\langle x,A_{ij}(t)x\rangle E_{ij}-(E_{ij}x)\cdot (E_{ij}x)^T\\
		&-tx\cdot (E_{ij}x)^T-t(E_{ij}x)\cdot x^T-t^2 x\cdot x^T+t\langle x,A_{ij}(t)x\rangle x\cdot x^T.
	\end{align*}
	to $x^\perp=T_xS^{n-1}$ span $\Sym^2 (x^\perp)$ for all $x\in S^{n-1}$, which implies the claim. Note that we can identify $\Sym^2(x^\perp)$ with the space of matrices $C\in \Sym^2(\R^n)$ that satisfy $Cx=0$.\\
	
	Fix $x\in S^{n-1}$ and assume that the restrictions of $B_{ij}(x)$, $1\le i\le j\le n$, to $x^\perp$ span a proper subspace of $\Sym^2(x^\perp)$. Then there exists a non-trivial $C\in \Sym^2(\R^n)$ with $Cx=0$ such that $\tr(CB_{ij}(x))=0$ for all $1\le i\le j\le n$. We may assume that $\|C\|_*= \tr(CE_{i_0j_0})>0$ for some $1\le i_0\le j_0\le n$. However, as $t=1+\frac{2}{c}$ and $\|E_{ij}\|=1$, this leads to the contradiction
	\begin{align*}
		0=&\tr(CB_{i_0j_0}(x))=\tr\left(C(E_{i_0j_0}\langle x,A_{i_0j_0}(t)x\rangle-(E_{i_0j_0}x)\cdot (E_{i_0j_0}x)^T)\right)\\
		=&\tr(CE_{i_0j_0})\langle x,A_{i_0j_0}(t)x\rangle-\langle E_{i_0j_0}x,CE_{i_0j_0}x\rangle\\
		\ge&\tr(CE_{i_0j_0}) (t-\|E_{i_0j_0}\|)|x|^2-\|C\|\cdot \|E_{i_0j_0}\|^2|x|^2\\
		=&\|C\|_*\left(t-1\right)-\|C\|\ge c\|C\|\left(t-1\right)-\|C\|=\left(c(t-1)-1\right)\|C\|=\|C\|>0.
	\end{align*}
	Thus the restrictions of these matrices have to span $\Sym^2(x^\perp)$. The claim follows.
	\end{proof}

	We will deduce Theorem \ref{theorem:finiteCombMixedVolumes} from the following result concerning finite linear combinations of valuations obtained from mixed area measures.
	\begin{theorem}
		\label{theorem:finiteCombMixedSurfaceArea}
		There exist $N:=\binom{n+1}{2}+1$ ellipsoids $\mathcal{E}_{1},\dots,\mathcal{E}_{N}$ such that for every $\mu\in\Val_k(\R^n)^{sm}\otimes \C$ there exist functions $g_{\alpha}\in C^\infty(S^{n-1})$, $\alpha\in \mathbb{N}^N$ with $|\alpha|=n-k-1$, such that
		\begin{align*}
			\mu(K)=\sum_{\substack{\alpha\in \mathbb{N}^N\\|\alpha|=n-k-1}}\int_{S^{n-1}}g_{\alpha}dS_{n-1}(K[k],\mathcal{E}_{1}[\alpha_1],\dots,\mathcal{E}_{N}[\alpha_N]).
		\end{align*}
		Moreover, if $\mu$ is even or odd, then $g_\alpha$ can be chosen to be even or odd respectively.
	\end{theorem}
	\begin{proof}
		Note that this expression is linear in $g_\alpha$, so we can replace $g_\alpha$ by its even or odd part for even or odd valuations respectively, which shows the last claim.\\
		Let $\mu\in\Val_k(\R^n)^{sm}\otimes\C$. By Corollary \ref{corollary:smoothValSeriesMixedVolumes}, there exists $F\in  C^\infty((S^{n-1})^{n-k})$ such that $\mu=\Theta_{n-k}(F)$.
		By Lemma \ref{lemma:SchwartzKernelDecompFunction}, we find functions $f_1^j,\dots,f^j_{n-k}\in C^\infty(S^{n-1})$ such that
		\begin{itemize}
			\item $F=\sum_{j=1}^{\infty}f_1^j\otimes\dots \otimes f_{n-k}^j$, and
			\item for all $l=(l_1,\dots,l_{n-k})\in\mathbb{N}^{n-k}$: 
			\begin{align}
				\label{eq:finiteSeriesCM}
				\sum_{i=1}^{\infty} \prod_{i=1}^{n-k}\|f_i^j\|_{C^{l_i}(S^{n-1})}\le C_{M_l}\|F\|_{C^{M_l}((S^{n-1})^k)}
			\end{align}
			for a constant $C_{M_l}$ independent of $F$, where $M_l=|l|+(n-k)(n-1)+1$.
		\end{itemize}
		In other words, the series converges in the $C^\infty$-topology. As $\Theta_{n-k}$ is continuous, this implies
		\begin{align*}
			\Theta_{n-k}(F)[K]=&\sum_{j=1}^\infty\Theta_{n-k}(f_1^j\otimes\dots\otimes f^j_{n-k})[K].
		\end{align*}
		If $K\in\mathcal{K}(\R^n)$ is smooth with strictly positive Gauss curvature, then the multilinearity of the mixed discriminant and the defining property of $\Theta_{n-k}$ from Corollary \ref{corollary:smoothValSeriesMixedVolumes} imply
		\begin{align*}
			\Theta_{n-k}(f_1^j\otimes\dots\otimes f^j_{n-k})[K]=\int_{S^{n-1}}f^j_1 D_{n-1}(D^2h_K[k], D^2f^j_2,\dots,D^2f^j_{n-k})d\mathcal{H}^{n-1}.
		\end{align*}
		By Lemma \ref{lemma:ChoiceEllipsoids}, there exist ellipsoids $\mathcal{E}_1,\dots,\mathcal{E}_N$ such that $D^2h_{\mathcal{E}_1},\dots, D^2h_{\mathcal{E}_N}$ span $\Sym^2T_xS^{n-1}$. In particular, for any point $x_0\in S^{n-1}$ we can find a neighborhood $U$ such that a selection of these functions generate a basis of $\Sym^2(T_xS^{n-1})$ for all $x\in U$. As $S^{n-1}$ is compact, we find a finite cover $U_i$, $1\le i\le m$, of subsets with this property. Denote the ellipsoids corresponding to the chosen basis on $U_i$ by $\mathcal{E}^i_{1},\dots, \mathcal{E}^i_{\binom{n}{2}}$, and let $\psi_i\in C^\infty(U_i)$ be a partition of unity subordinate to the open cover $(U_i)_{1\le i\le m}$. As $\psi_{i}$ is compactly supported in $U_i$, we find smooth sections $\Psi^i_{s}\in C^\infty_c(U_i,\Sym^2(TS^{n-1})^*)$ such that
		\begin{align*}
			A=\sum_{s=1}^{\binom{n}{2}}\Psi^i_s|_x(A)D^2h_{\mathcal{E}^i_s}(x)\quad\text{for}~A\in\Sym^2T_xS^{n-1}
		\end{align*}
		for all $x$ belonging to a neighborhood of $\supp\psi_{i}$. 
		Using the multilinearity of the mixed discriminant, we thus obtain for $K$ smooth with strictly positive Gauss curvature
		\begin{align}
			\notag
			&\Theta_{n-k}(F)[K]=\sum_{j=1}^\infty\sum_{i_2,\dots,i_{n-k}}^m\int_{S^{n-1}}f^j_1 D_{n-1}(D^2h_K[k],\psi_{i_2} D^2f^j_2,\dots,\psi_{i_{n-k}}D^2f^j_{n-k})d\mathcal{H}^{n-1}\\
			\notag
			=&\sum_{j=1}^\infty\sum_{i_2,\dots,i_{n-k}}^m\sum_{s_2,\dots,s_{n-k}}^{\binom{n}{2}}\int_{S^{n-1}}f^j_1\prod_{l=2}^{n-k}\psi_{i_l} \Psi^{i_l}_{s_l}(D^2f^j_l) D_{n-1}(D^2h_K[k],D^2h_{\mathcal{E}^{i_2}_{s_2}},\dots,D^2h_{\mathcal{E}^{i_{n-k}}_{s_{n-k}}})d\mathcal{H}^{n-1}\\
			\label{equation:RepresentationMuMixedSurfaceArea}
			=&\sum_{j=1}^\infty\sum_{i_2,\dots,i_{n-k}}^m\sum_{s_2,\dots,s_{n-k}}^{\binom{n}{2}}\int_{S^{n-1}}f^j_1\prod_{l=2}^{n-k}\psi_{i_l} \Psi^{i_l}_{s_l}(D^2f^j_l) dS_{n-1}(K[k],\mathcal{E}^{i_2}_{s_2},\dots,\mathcal{E}^{i_{n-k}}_{s_{n-k}}).
		\end{align}
		Using the Leibniz rule, we see that for every $M\in\mathbb{N}$ there exists $D_M>0$ such that
		\begin{align*}
			\left\|\psi_{i_l} \Psi^{i_l}_{s_l}(D^2f^j_l)\right\|_{C^M(S^{n-1})}\le D_M \|f^j_l\|_{C^{M+2}(S^{n-1})}.
		\end{align*}
		Using the Leibniz rule again, we thus obtain
		\begin{align*}
			\sum_{j=1}^{\infty}\left\|f^j_1\prod_{l=2}^{n-k}\psi_{i_l} \Psi^{i_l}_{s_l}(D^2f^j_l)\right\|_{C^M(S^{n-1})}\le& c_{M,n,k}\sum_{j=1}^{\infty}\|f^j_1\|_{C^M(S^{n-1})}\prod_{l=2}^{n-k}\left\|\psi_{i_l} \Psi^{i_l}_{s_l}(D^2f^j_l)\right\|_{C^M(S^{n-1})}\\
			\le&c_{M,n,k}D_M^{n-k}\sum_{j=1}^\infty\|f^j_1\|_{C^M(S^{n-1})}\prod_{l=2}^{n-k}\|f_l^j\|_{C^{M+2}(S^{n-1})},
		\end{align*}
		for some constant $c_{M,n,k}$, which is finite due to \eqref{eq:finiteSeriesCM}. Thus the series
		\begin{align*}
			g_{i_1,\dots,i_m,s_1,\dots,s_{n-k}}:=	\sum_{j=1}^{\infty}f^j_1\prod_{l=2}^{n-k}\psi_{i_l} \Psi^{i_l}_{s_l}(D^2f^j_l)
		\end{align*}
		converges in the $C^\infty$ topology to a smooth function $g_{i_1,\dots,i_m,s_1,\dots,s_{n-k}}$. In particular, the series converges uniformly on $S^{n-1}$, so from \eqref{equation:RepresentationMuMixedSurfaceArea} we obtain for $K$ smooth with strictly positive Gauss curvature
		\begin{align*}
			\mu(K)=\Theta_{n-k}(F)[K]=\sum_{i_2,\dots,i_{n-k}}^m\sum_{s_2,\dots,s_{n-k}}^{\binom{n}{2}}\int_{S^{n-1}}g_{i_1,\dots,i_m,s_1,\dots,s_{n-k}} dS_{n-1}(K[k],\mathcal{E}^{i_2}_{s_2},\dots,\mathcal{E}^{i_{n-k}}_{s_{n-k}}),
		\end{align*}
		where the ellipsoids $\mathcal{E}^{i}_{s}$ do not depend on $\mu$. As both sides depend continuously on $K$, this representation holds for all $K\in\mathcal{K}(\R^n)$. Now the claim follows by combining the functions with indices corresponding to the same mixed area measures $dS_{n-1}(\cdot[k],\mathcal{E}_1[\alpha_1],\dots,\mathcal{E}_N[\alpha_N])$, $\alpha\in \mathbb{N}^N$ with $|\alpha|=n-k-1$.
	\end{proof}
	
	\begin{proof}[Proof of Theorem \ref{theorem:finiteCombMixedVolumes}]
		This follows from Theorem \ref{theorem:finiteCombMixedSurfaceArea} in combination with \eqref{eq:relationMixedVolumeMixedSurfaceAreaMeasure}: If $\mu$ is a real-valued valuation, then we may assume that the functions $g_\alpha$ in Theorem \ref{theorem:finiteCombMixedSurfaceArea} are real-valued as well, as we can replace them by their real part due to linearity. We can then choose smooth convex bodies $L^\pm_\alpha\in \mathcal{K}(\R^n)$ with strictly positive Gauss curvature such that $g_\alpha=h_{L^+_\alpha}-h_{L^-_\alpha}$, for example by fixing a ball $L^-_\alpha$ such that $g_\alpha+h_{L^-_\alpha}$ is convex with positive definite Hessian, that is, such that it is the support function $h_{L^+_\alpha}$ of a smooth convex body with strictly positive Gauss curvature. Exchanging the role of $L^+_\alpha$ and $L^-_\alpha$, we see that we can always assume that one of these bodies is an ellipsoid. If $\mu$ is in addition an even valuation, then $g_\alpha$ can be chosen to be even by Theorem \ref{theorem:finiteCombMixedSurfaceArea}, which obviously implies that $h_{L^\pm_\alpha}=g_\alpha+h_{L^\mp_\alpha}$ is even as well. Thus the remaining bodies can be chosen to be origin symmetric.
	\end{proof}
	
\section{Finite linear combinations of mixed volumes in the general case}
	\label{section:ContinuousCase}
	In this section we will construct valuations of degree $1\le k\le n-1$ that do not admit a representation as a finite linear combination of mixed volumes. We start with the following construction, which is due to Goodey and Weil \cite{GoodeyWeilDistributionsvaluations1984}. The version below can be found in \cite{AleskerIntroductiontheoryvaluations2018} Proposition 7.0.2.\\
	We refer to \cite{Hoermanderanalysislinearpartial2003} for a background on distributions and only recall that a distribution $T:C^\infty(S^{n-1})\rightarrow\R$ is said to have order at most $m\in\mathbb{N}$ if there exists a constant $C>0$ such that $|T(\phi)|\le C\|\phi\|_{C^m(S^{n-1})}$ for all $\phi\in C^\infty(S^{n-1})$.
	\begin{proposition}
		For every $\mu\in \Val_1(\R^n)$ there exists a unique distribution $\GW(\mu)$ on $S^{n-1}$ of order at most $2$ such that 
		\begin{align*}
			\GW(\mu)[h_K]=\mu(h_K)
		\end{align*}
		for all smooth convex bodies $K\in\mathcal{K}(\R^n)$ with strictly positive Gauss curvature.
	\end{proposition}	
	Recall that a valuation $\mu\in \Val_1(\R^n)$ is called \emph{uniformly continuous} if $|\mu(K)|\le C\|h_K\|_\infty$. It is easy to see that this is the case if and only if $\GW(\mu)$ is of order $0$. Using this notion, linear combination of mixed volumes admit the following complete characterization for $k=1$. 
	\begin{theorem}[McMullen \cite{McMullenContinuoustranslationinvariant1980} Theorems 3 and 4]
		\label{theorem:UniformlyContValuationsMV}
		$\mu\in\Val_1(\R^n)$ is a linear combination of mixed volumes if and only if $\mu$ is uniformly continuous, that is, if and only if $\GW(\mu)$ is of order $0$.
	\end{theorem}
	Using this description, it is not difficult to obtain examples of elements of $\Val_1(\R^n)$ that are not linear combinations of mixed volumes. We will need the following specific example. Recall that the $k$-th area measure of a convex body $K$ is given by $S_k(K):=S(K[k],B_1(0)[n-k-1])$, where $B_1(0)$ denotes the unit ball in $\R^n$ centered at the origin.
	\begin{lemma}
		\label{lemma:nonUniformlyContinuousValuationDeg1}
		Let $\psi\in C^\infty([-1,1])$ be a function with $\psi(t)=0$ for $|t|\ge \frac{2}{3}$ and $\psi(t)=1$ for $|t|\le \frac{1}{3}$. Define $f\in C([-1,1])$ by $f(t)=\sqrt{|t|}(1-t^2)^{-\frac{n-3}{2}}\psi(t)$. The valuation $\mu\in \Val_1(\R^n)$ given by 
		\begin{align*}
			\mu(K):=\int_{S^{n-1}}f(x_n)dS_1(K,x), \quad K\in\mathcal{K}(\R^n),
		\end{align*} 
		is not uniformly continuous.
	\end{lemma}
	\begin{proof}
		We have to show that $\GW(\mu)$ is not bounded with respect to the supremum norm $\|\cdot\|_\infty$. Let $\phi\in C^\infty_c((-1,1))$ be a function with $\supp\phi\subset [\frac{\epsilon}{2},\frac{1}{3}]$ for some $\epsilon>0$. Set $\tilde{\phi}(x):=\phi(x_n)$. Using the definition of the mixed area measure, it is easy to see that
		\begin{align*}
			\GW(\mu)[\tilde{\phi}]=\int_{S^{n-1}}f(x_n)\left[\Delta_{S^{n-1}}\tilde{\phi}(x)+(n-1)\tilde{\phi}(x)\right]d\mathcal{H}^{n-1}(x).
		\end{align*}
		By a change to spherical cylinder coordinates, this reduces to
		\begin{align*}
			\GW(\mu)[\tilde{\phi}]=&\omega_{n-2}\int_{-1}^1f(t)\left[(1-t^2)\phi''(t)-(n-1)t\phi'(t)+(n-1)\phi(t)\right](1-t^2)^{\frac{n-3}{2}}dt\\
			=&\omega_{n-2}\int_0^1\sqrt{t}\left[(1-t^2)\phi''(t)-(n-1)t\phi'(t)+(n-1)\phi(t)\right]dt.
		\end{align*}
		Integrating by parts, we thus obtain
		\begin{align*}
				\GW(\mu)[\tilde{\phi}]=&\omega_{n-2}\int_0^1\left[\sqrt{t}(1-t^2)\right]''\phi(t)+\frac{3}{2}(n-1)\sqrt{t}\phi(t)+(n-1)\phi(t)dt\\
				=&\omega_{n-2}\int_0^1-\frac{1}{4\sqrt{t}^3}\phi(t)-\frac{15}{4}\sqrt{t}\phi(t)+\frac{3}{2}(n-1)\sqrt{t}\phi(t)+(n-1)\phi(t)dt.
		\end{align*}
		Except for the first expression, these terms can obviously be bounded by a constant multiple of $\|\phi\|_\infty$. Consider the term
		\begin{align*}
			T(\phi):=\int_0^1\frac{\phi(t)}{\sqrt{t}^3}dt.
		\end{align*}
		Assume that $\phi$ is non-negative with $\|\phi\|_\infty\le 1$ and satisfies $\phi(t)=1$ for $t\in (\epsilon,4\epsilon)$. Then 
		\begin{align*}
			T(\phi)\ge \int_\epsilon^{2\epsilon}\frac{1}{\sqrt{t}^3}dt=\left[-\frac{2}{\sqrt{t}}\right]_\epsilon^{4\epsilon}=\frac{1}{\sqrt{\epsilon}}.
		\end{align*}
		As this diverges for $\epsilon\rightarrow0$, $\GW(\mu)[\tilde{\phi}]$ cannot by bounded by a constant multiple of $\|\tilde{\phi}\|_\infty$ for all $\phi\in C_c^\infty((-1,1))$. Thus $\mu$ is not uniformly continuous.
	\end{proof}

	\begin{corollary}
		Mixed volumes span a proper subspace of $\Val_k(\R^n)$ for all $1\le k\le n-1$.
	\end{corollary}
	\begin{proof}
		It is sufficient to construct a $k$-homogeneous valuation that is not a linear combination of mixed volumes for all $1\le k\le n-1$.\\
		Let $f\in C([-1,1])$ denote the function defined in Lemma \ref{lemma:nonUniformlyContinuousValuationDeg1}. For $k=1$, the valuation $\mu_1(K):=\int_{S^{n-1}}f(x_n)dS_1(K,x)$ is not uniformly continuous by Lemma \ref{lemma:nonUniformlyContinuousValuationDeg1} and thus not a linear combination of mixed volumes by Theorem \ref{theorem:UniformlyContValuationsMV}.\\
		Now let $2\le k\le n-1$ and consider the valuation $\mu_k\in\Val_k(\R^n)$ defined by
		\begin{align*}
			\mu_k(K)=\int_{S^{n-1}}f(x_n)dS_k(K,x).
		\end{align*}
		If this was a linear combination of mixed volumes, then this would also hold for the valuation
		\begin{align*}
			K\mapsto \frac{d^{k-1}}{dt^{k-1}}\Big|_0\mu_k(K+tB_1(0))=(k-1)!\int_{S^{n-1}}f(x_n)dS_1(K,x),
		\end{align*}
		which is a contradiction. Thus $\mu_k$ can not be represented as a finite linear combination of mixed volumes.
	\end{proof}
	
\section{Questions}
	\label{section:questions}
	In this section we collect some questions related to Theorem \ref{theorem:finiteCombMixedVolumes}. Let $\mathcal{L}_{k,m}\subset\Val_k(\R^n)$ denote the set of linear combinations of at most $m$ mixed volumes. Corollary \ref{theorem:densityFiniteComb} implies that $\mathcal{L}_{k,m}$ is dense in $\Val_k(\R^n)$ for $m\ge 2\binom{\binom{n+1}{2}+n-k-1}{n-k-1}$.
	\begin{question}
		What is the minimal number $N_{n,k}$ such that $\mathcal{L}_{k,N_{n,k}}\subset\Val_k(\R^n)$ is dense? 
	\end{question}
	Note that for $k=1$ and $k=n-1$ the answer is $N_{n,k}=2$ for $n\ge 2$. This follows from results due to Goodey and Weil \cite{GoodeyWeilDistributionsvaluations1984} for $k=1$ and McMullen \cite{McMullenContinuoustranslationinvariant1980} for $k=n-1$.\\
	
	The space $\Val(\R^n)^{sm}$ is equipped with a number of algebraic structures, for example the Alesker product \cite{Aleskermultiplicativestructurecontinuous2004} and the Bernig-Fu convolution \cite{BernigFuConvolutionconvexvaluations2006}, which are in turn related by the Alesker-Fourier transform \cite{AleskerFouriertypetransform2011}. Both of these products admit an interpretation in terms of mixed volumes. In particular, Theorem \ref{theorem:finiteCombMixedVolumes} implies that this already provides a complete description of these structures. However, in contrast to the Bernig-Fu convolution, the description of the Alesker product involves mixed volumes defined on $\R^{n}\times\R^n$. 
	\begin{question}
		\label{question:DescrAleskerProductFourier}
		Find an explicit description of the Alesker product and the Alesker-Fourier transform in terms of linear combinations of smooth mixed volumes on $\R^n$. 
	\end{question}
	
	Invariant valuations play an important role in integral geometry. For the unitary group $\U(n)$ acting on $\C^n$, the space $\Val(\C^n)^{\U(n)}$ of of $\U(n)$-invariant valuations in $\Val(\C^n)$ was described by Alesker \cite{AleskerHardLefschetztheorem2003}, Fu \cite{FuStructureunitaryvaluation2006}, and Bernig and Fu \cite{BernigFuHermitianintegralgeometry2011}. In particular, there exist a number of different bases of this space relying on different integral geometric and differential geometric constructions. As all of these valuations are smooth, Theorem \ref{theorem:finiteCombMixedVolumes} implies that they admit a description in terms of mixed volumes. It would be interesting to obtain a description of this space and its algebra structure in terms suitable combinations of mixed volumes. This leads to the following question, which is due to Kotrbat\'y.
	\begin{question}
		\label{question:GeneratorValUnitary}
		Describe the generators of the algebra $\Val(\C^n)^{\U(n)}$ in terms of linear combinations of mixed volumes.
	\end{question}
	Note that $\Val(\C^n)^{\U(n)}$ is generated as an algebra by two elements. One of these elements is an intrinsic volume (either of degree $1$ for the Alesker product or of degree $2n-1$ for the Bernig-Fu convolution), which can therefore easily be described in terms of mixed volumes. Question \ref{question:GeneratorValUnitary} thus asks for a suitable representation for the second generator.

	\bibliography{../../library/library.bib}

\begin{thebibliography}{10}

\bibitem{AleskerDescriptiontranslationinvariant2001}
Semyon Alesker.
\newblock Description of translation invariant valuations on convex sets with
  solution of {P}. {M}c{M}ullen's conjecture.
\newblock {\em Geom. Funct. Anal.}, 11(2):244--272, 2001.

\bibitem{AleskerHardLefschetztheorem2003}
Semyon Alesker.
\newblock Hard {L}efschetz theorem for valuations, complex integral geometry,
  and unitarily invariant valuations.
\newblock {\em J. Differential Geom.}, 63(1):63--95, 2003.

\bibitem{Aleskermultiplicativestructurecontinuous2004}
Semyon Alesker.
\newblock The multiplicative structure on continuous polynomial valuations.
\newblock {\em Geom. Funct. Anal.}, 14(1):1--26, 2004.

\bibitem{AleskerFouriertypetransform2011}
Semyon Alesker.
\newblock A {F}ourier-type transform on translation-invariant valuations on
  convex sets.
\newblock {\em Israel J. Math.}, 181:189--294, 2011.

\bibitem{AleskerIntroductiontheoryvaluations2018}
Semyon Alesker.
\newblock {\em Introduction to the theory of valuations}, volume 126 of {\em
  CBMS Regional Conference Series in Mathematics}.
\newblock Conference Board of the Mathematical Sciences, Washington, DC; by the
  American Mathematical Society, Providence, RI, 2018.

\bibitem{AleskerFuIntegralgeometryvaluations2014}
Semyon Alesker and Joseph H.~G. Fu.
\newblock {\em Integral geometry and valuations}.
\newblock Advanced Courses in Mathematics. CRM Barcelona.
  Birkh\"{a}user/Springer, Basel, 2014.

\bibitem{BernigAlgebraicintegralgeometry2011}
Andreas Bernig.
\newblock Algebraic integral geometry.
\newblock In {\em Global differential geometry}, volume~17 of {\em Springer
  Proc. Math.}, pages 107--145. Springer, Heidelberg, 2012.

\bibitem{BernigFuConvolutionconvexvaluations2006}
Andreas Bernig and Joseph H.~G. Fu.
\newblock Convolution of convex valuations.
\newblock {\em Geom. Dedicata}, 123:153--169, 2006.

\bibitem{BernigFuHermitianintegralgeometry2011}
Andreas Bernig and Joseph H.~G. Fu.
\newblock Hermitian integral geometry.
\newblock {\em Ann. of Math. (2)}, 173(2):907--945, 2011.

\bibitem{CasselmanCanonicalextensionsHarish1989}
William Casselman.
\newblock Canonical extensions of {H}arish-{C}handra modules to representations
  of {$G$}.
\newblock {\em Canad. J. Math.}, 41(3):385--438, 1989.

\bibitem{FuStructureunitaryvaluation2006}
Joseph H.~G. Fu.
\newblock Structure of the unitary valuation algebra.
\newblock {\em J. Differential Geom.}, 72(3):509--533, 2006.

\bibitem{GaskproofSchwartzskernel1961}
H.~Gask.
\newblock A proof of {S}chwartz's kernel theorem.
\newblock {\em Math. Scand.}, 8:327--332, 1960.

\bibitem{GoodeyWeilDistributionsvaluations1984}
Paul Goodey and Wolfgang Weil.
\newblock Distributions and valuations.
\newblock {\em Proc. London Math. Soc. (3)}, 49(3):504--516, 1984.

\bibitem{HadwigerVorlesungenuberInhalt1957}
Hugo Hadwiger.
\newblock {\em Vorlesungen \"{u}ber {I}nhalt, {O}berfl\"{a}che und
  {I}soperimetrie}.
\newblock Springer, Berlin-G\"{o}ttingen-Heidelberg, 1957.

\bibitem{HarishChandraRepresentationssemisimpleLie1953}
Harish-Chandra.
\newblock Representations of a semisimple {L}ie group on a {B}anach space. {I}.
\newblock {\em Trans. Amer. Math. Soc.}, 75:185--243, 1953.

\bibitem{Hoermanderanalysislinearpartial2003}
Lars H\"{o}rmander.
\newblock {\em The analysis of linear partial differential operators. {I}}.
\newblock Classics in Mathematics. Springer, Berlin, 2003.

\bibitem{McMullenValuationsEulertype1977}
Peter McMullen.
\newblock Valuations and {E}uler-type relations on certain classes of convex
  polytopes.
\newblock {\em Proc. London Math. Soc. (3)}, 35(1):113--135, 1977.

\bibitem{McMullenContinuoustranslationinvariant1980}
Peter McMullen.
\newblock Continuous translation-invariant valuations on the space of compact
  convex sets.
\newblock {\em Arch. Math. (Basel)}, 34(4):377--384, 1980.

\bibitem{SchneiderConvexbodiesBrunn2014}
Rolf Schneider.
\newblock {\em Convex bodies: the {B}runn-{M}inkowski theory}, volume 151 of
  {\em Encyclopedia of Mathematics and its Applications}.
\newblock Cambridge University Press, Cambridge, expanded edition, 2014.

\end{thebibliography}

\Addresses
\end{document}